\documentclass{amsart}
\usepackage[utf8]{inputenc}
\usepackage{amsmath}
\usepackage{mathtools}
\usepackage{amsaddr}
\usepackage{graphicx}
\usepackage{amssymb}
\usepackage{mathrsfs}
\usepackage[nobysame]{amsrefs}
\usepackage{hyperref}

\newtheorem{definition}{Definition}
\newtheorem{theorem}{Theorem}[section]
\newtheorem{lemma}[theorem]{Lemma}
\newtheorem{remark}{Remark}

\newtheorem*{theorem*}{Theorem}
\newtheorem{corollary}[theorem]{Corollary}
\numberwithin{theorem}{section}
\numberwithin{definition}{section}

\title{A geometric dynamical system with relation to billiards}
\author{Samuel Everett}
\email{same@uchicago.edu}

\begin{document}

\begin{abstract}
    We introduce a geometric dynamical system where iteration is defined as a cycling composition of different maps acting on a space composed of three or more lines in $\mathbb{R}^2$.  This system is motivated by the dynamics of iterated function systems, as well as billiards with modified reflection laws.  We provide conditions under which this dynamical system generates periodic orbits, and use this result to prove the existence of closed nonsmooth curves over $\mathbb{R}^2$ which satisfy particular structural constraints with respect to a space of intersecting lines in the plane.
\end{abstract}

\maketitle

\section{Introduction}

The theory of mathematical billiards in polygons concerns the uniform motion of a point mass (billiard) in a polygonal plane domain, with elastic reflections off the boundary according to the mirror law of reflection.  One may also consider billiards with modified reflection laws, so that the angle of reflection is some function of the angle of incidence (see e.g., \cites{arroyo, arroyo2, magno, magno2, gtroub} and the references therein), and tiling billiards, where trajectories refract through planar tilings (see \cites{davis, davis2}).
It is natural to ask whether there exists a periodic billiard trajectory.
However, this question quickly becomes difficult to answer, and is open in the case of polygonal billiards obeying the mirror law of reflection (see \emph{Problem 10} in \cite{gutkin}, and \cite{gutkint} for a survey), although intense study has yielded a great deal of deep theorems towards the problem \cites{masur, schwartz, galperin, troub, halb}.

The purpose of this paper is to study a relaxed version of this problem, where the reflection rule is highly irregular: trajectories reflect off or refract through lines as a function of the line they are incident to.  Physically, consider a cracked pane of glass in which various materials are allowed between the cracks, and a laser beam passes in from one side of the pane.
To study such a system and consider existence of periodic trajectories, we introduce a dynamical system analogous to an iterated function system (see \cite{hutch} for review) where the defining collection of contraction mappings are geometrically characterized over lines in the plane, and composed in a fixed, cycling order.  We prove Theorem \ref{thm1}, which asserts the existence of nonsmooth closed curves satisfying particular geometric constraints with respect to a space of lines on the plane.  In fact, such closed curves can coincide with periodic billiard trajectories with standard reflection laws.  We state this result after giving some notation.

Let $X_m \subset \mathbb{R}^2$ denote a union of $m \geq 3$ nonconcurrent lines in $\mathbb{R}^2$ with at least one line not parallel with the others.  Assign each line in $X_m$ a unique label $L_i$, $i \in \{1, 2, ..., m\}$.  Let $p_1, p_2, ..., p_n$, be a sequence of $n\geq m$ points in $X_m$ such that consecutive points, including $p_n$ and $p_1$, are distinct, and if $p_k \in L_i$, then $p_{k+1} \in L_j$, $i\not= j$ (with the convention that $p_{n+1} \equiv p_1$).  Join consecutive pairs of such points, including $p_n$ and $p_1$, with line segments to construct a \emph{closed curve} $\Gamma$ over $X_m$.  Traversal of a closed curve in a fixed direction allows for construction of an \textit{incidence angle sequence $\theta_1, \theta_2, ..., \theta_n$ with respect to a line sequence $L_{a_1}, L_{a_2}, ..., L_{a_n}$}, by taking the acute or right angle $\theta_i$ between each segment of the closed curve and a line with label $L_{a_i}$ it is incident to, with respect to the traversal direction.  Refer to Figure \ref{fig2} for visual demonstration.  

\begin{figure}[b]
    \centering
    \includegraphics[scale=.21]{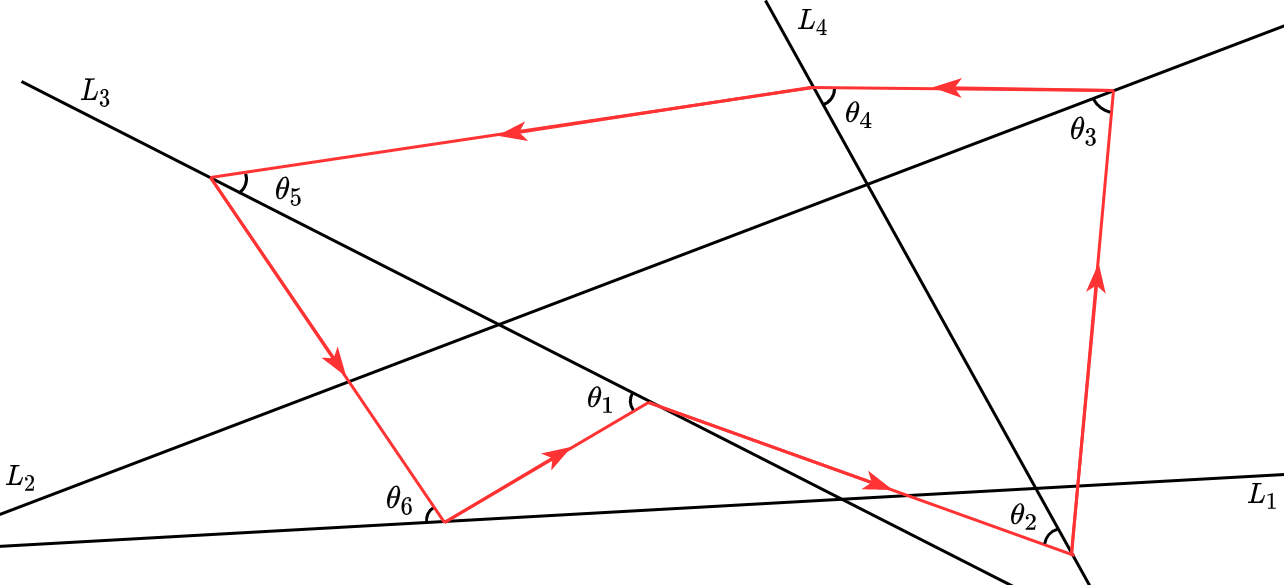}
    \caption{The values $\theta_1, \theta_2, \theta_3, \theta_4, \theta_5, \theta_6$  compose an incidence angle sequence with respect to the line sequence $L_3, L_4, L_2, L_4, L_3, L_1$.}
    \label{fig2}
\end{figure}

\begin{theorem}\label{thm1}
For any space $X_m$ with labeled lines, let $L_{a_1}, L_{a_2}, ..., L_{a_n}$, $n\geq m$, be a sequence of line labels such that no two consecutive labels are the same, including $L_{a_n}$ and $L_{a_1}$, and each of the $m$ labels occur at least once in the sequence.  Then for Lebesgue almost every $(\theta_1, \theta_2,...,\theta_n) \in (0, \pi/2]^n\subset \mathbb{R}^2$, there exists a closed curve $\Gamma$ over $X_m$ that admits an incidence angle sequence $\theta_1, \theta_2, ..., \theta_n$ with respect to the line sequence $L_{a_1}, L_{a_2}, ..., L_{a_n}$ when traversed in a fixed direction.
\end{theorem}

In the case where a closed curve is strictly contained within a polygon formed by the intersecting lines composing $X_m$, the closed curve does not cross over any lines in the space, so the angles of incidence implicitly define angles of reflection. Hence, when the parameters of the closed curve are such that the angles of incidence equal the angles of reflection, or the angles of reflection are a function of the angles of incidence, the closed curve corresponds to a periodic billiard trajectory obeying the mirror law of reflection or some modified reflection law.  However, all closed curves need not correspond to a periodic billiard trajectory.  

This paper is organized into two main parts, separated by study of two related dynamical systems.  In the first, from Sections \ref{sec:prelim} through \ref{sec:closed}, we define a dynamical system that provides controllable and predictable behavior, which we use to prove Theorem \ref{thm1}.  In the second part, from Sections \ref{sec:dist} through \ref{sec:periodic}, we redefine components of the dynamical system given in Section \ref{sec:prelim} in a way that introduces discontinuities.  The introduction of such discontinuities leads to far more complex dynamics that shares characteristics with piecewise isometries, the farthest point map on compact metric spaces, and generalizes \cite{everett}.  We prove a theorem that asserts orbits of this system are asymptotically stable when particular geometric conditions are satisfied.

\subsection*{Acknowledgements.}  The author would like to thank Nikhil Krishnan for the feedback in the early stages of this work.  The author is also grateful to the anonymous reviewer of this article, whose detailed feedback substantially improved the quality of this paper.

\section{Preliminaries}\label{sec:prelim}

Let $L_i, L_j$ label distinct lines in the space $X_m$.  Then, for every $x \in L_i$ we may determine two lines, $\mathscr{L}_0, \mathscr{L}_1$ such that $\{x\} = \mathscr{L}_0 \cap \mathscr{L}_1$ and $\mathscr{L}_0, \mathscr{L}_1$ intersect with line $L_j$ at an angle $\theta$ in $(0, \pi/2)$, with intersection points $z_0$ and $z_1$ in $L_j$, respectively.  For a visual demonstration, refer to Figure \ref{figRuleDef}.  We call $z_0, z_1$ the \textit{orientation 0 and 1, angle $\theta$ projections of $x$ onto $L_j$}.  If $\theta = \pi/2$, then we call the line intersection point $z$ the \textit{perpendicular projection of $x$ onto $L_j$}.  In the case when $x \in L_i \cap L_j$, the projection of $x$ onto $L_i$ or $L_j$ is simply $x$ itself.

\begin{figure}
    \centering
    \includegraphics[scale=0.25]{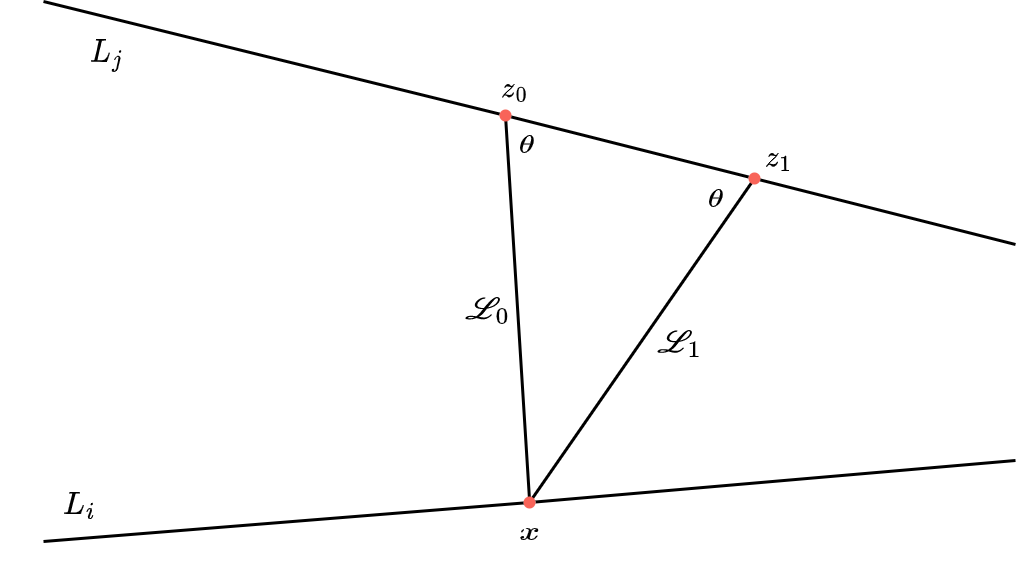}
    \caption{An illustration of orientation 0 and 1, angle $\theta$ projections of $x$ onto points $z_0, z_1$ in $L_j$. Note that an orientation $o \in \{0, 1\}$ corresponds with a choice in line $\mathscr{L}_o$.}
    \label{figRuleDef}
\end{figure}

\begin{definition}
Let $\theta \in (0, \pi/2]$, $o \in \{0, 1\}$, and $i \in \{1, ..., m\}$.  We call a mapping $r:X_m \rightarrow X_m$ a \emph{rule}, if $r(x)$ is an angle $\theta$, orientation $o$ projection of $x \in X_m$ onto a line $L_i$ in $X_m$, so $r(X_m)=L_i$.  We may also notate rules as $r(x; \theta, o, L_i)$ to make the parameters explicit.
\end{definition}

When the rule projection angle $\theta = \pi/2$, we simply write $r(x; \theta, L_i)$ when notating rules as there is only one possible orientation.   Figure \ref{fig1} provides visual demonstration of the composition of two rules, $r_1(x) \coloneqq r_1(x; \theta_1, 1, L_2)$ and $r_2(x) \coloneqq r_2(x; \theta_2, 0, L_3)$ over a point $x_0 \in L_1 \subset X_3$, so that 
\[
r_1(x_0) = x_1 \text{ and } r_2(r_1(x_0)) = x_2.
\]

\begin{figure}
    \centering
    \includegraphics[scale=0.35]{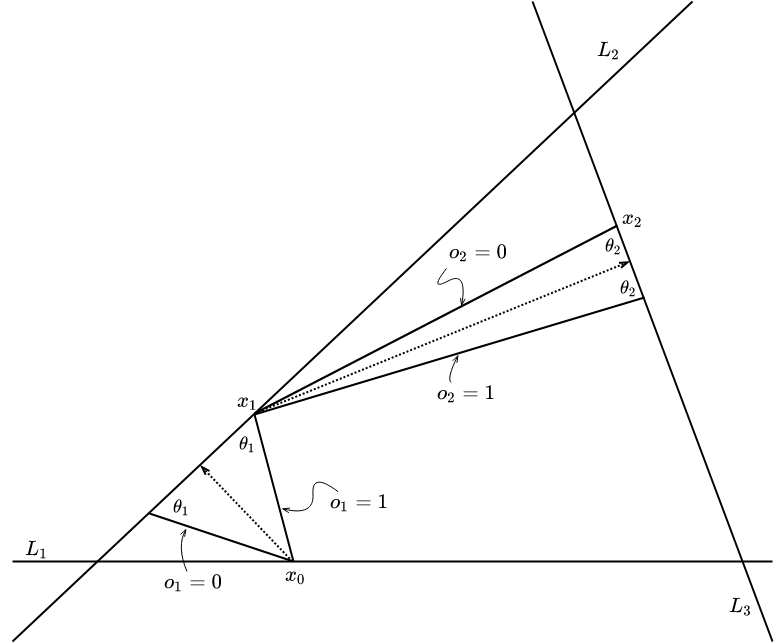}
    \caption{An illustration of the composition of two rules, $r_1(x; \theta_1, 1, L_2)$ and $r_2(x; \theta_2, 0, L_3)$ over a point $x_0 \in L_1$, so that $r_1(x_0) = x_1$, and $r_2(x_1) = x_2$.  The figure also shows the two orientation options for each rule, and the dotted line corresponds to the $\theta=\pi/2$ case for each rule.}
    \label{fig1}
\end{figure}

We require rule orientation to be defined in a predictable and consistent way, so that it is never ambiguous which projection points correspond to which orientation value.  For this paper, we choose a natural and mathematically convenient convention where the orientation 0 and 1 projection points under a rule $r(x)$ are the ``left" and ``right" points, ``from the perspective of $x$".  Figures \ref{fig1} and \ref{fig4} demonstrate this convention.

\begin{figure}
    \centering
    \includegraphics[scale=0.27]{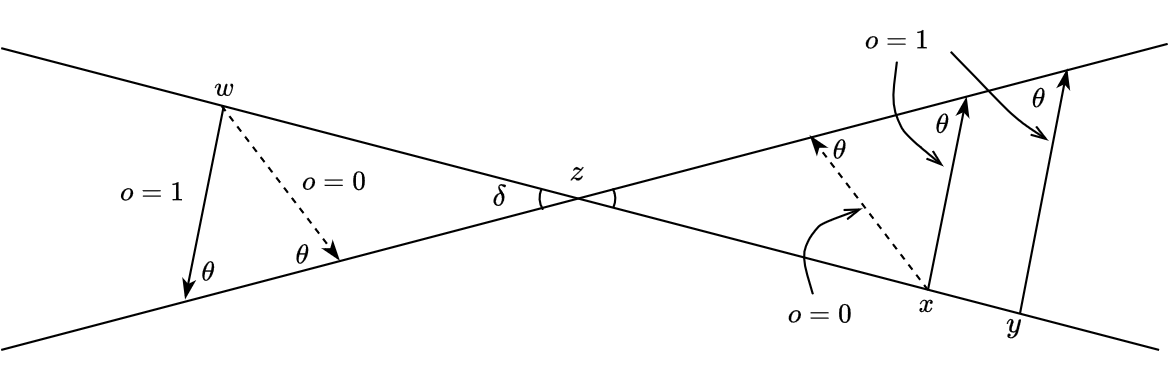}
    \caption{How orientation is preserved from translating $x$ to $w$ across line intersection point $z$, under a rule with projection angle $\theta$.  Note how the projection lines corresponding with a fixed orientation are antiparallel across line intersection point ($z$), and always map opposite the same line intersection angle ($\delta$).}
    \label{fig4}
\end{figure}

We define a \emph{rule sequence} associated to a space $X_m$ to be a sequence of $n \geq m \geq 3$ rules, denoted $\{r_i\}_{i=1}^n$, with the restriction that consecutive rules in the rule sequence, including $r_1$ and $r_n$, cannot map onto the same line in $X_m$.  Furthermore, we require each line in $X_m$ be mapped onto by at least one of the rules in an associated rule sequence.

\begin{definition}
An \emph{$n$-rule map} $T_n: X_m \rightarrow X_m$ is defined to be a cycling composition of $n\geq 3$ rules in an associated \emph{defining rule sequence} $\{r_i\}_{i=1}^n$.  That is, if $\{r_i\}_{i=1}^n$ is a sequence of rules defining $n$-rule map $T_n$,  then for $x \in X_m$, define iteration of $T_n$ so that
\[
T_n(T_{n}^{n+1}(x)) = T_{n}^{n+2}(x) = r_2(r_1(r_n(...r_2(r_1(x))))).
\]  
\end{definition}

Figure \ref{figMapIter} gives a visual example of iterating a $3$-rule map over $X_3$.

\begin{figure}
    \centering
    \includegraphics[scale=.27]{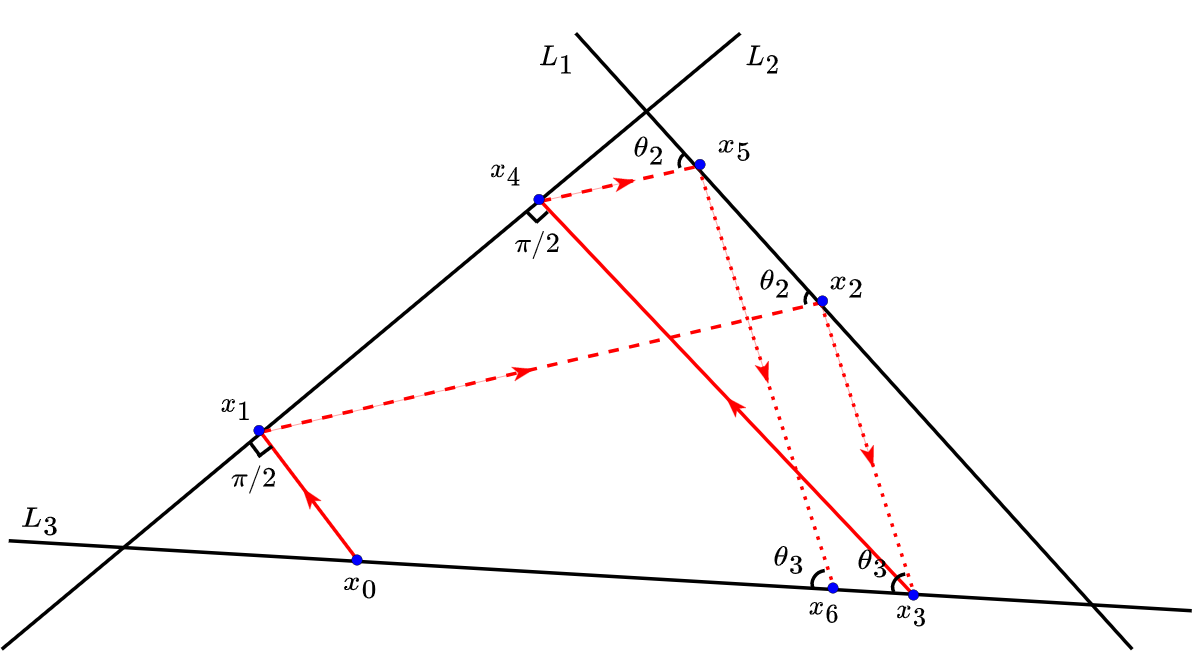}
    \caption{Demonstration of six iterations of a $3$-rule map $T_3$ over $X_3$, where $T_3(x_0) = x_1$, and $T_3^6(x_0) = x_6$.  The defining rule sequence of $T_3$ is $(r_1(x;\pi/2, L_2)$, $r_2(x;\theta_2, 1, L_1)$, $r_3(x;\theta_3, 0, L_3))$.  The solid lines correspond to rule $r_1$, the dashed lines to $r_2$, and the dotted lines to $r_3$.}
    \label{figMapIter}
\end{figure}

Unless otherwise stated, the pair $(X_m, T_n)$ denotes a dynamical system.  For a point $x \in X_m$, we let $\mathcal{O}(x)$ denote the \emph{orbit of $x$} under $n$-rule map $T_n$, so that 
\[
\mathcal{O}(x) \coloneqq \{x, T_n(x), T_n^2(x),...\}.
\]
We call an $n$-rule map \emph{redundant} if there exists a length $n'$ rule sequence with $m \leq n' < n$, such that for all $x \in X_m$, the orbit of $x$ under the $n'$-rule map is equal to the orbit of $x$ under the $n$-rule map.  For the purpose of this paper we assume all $n$-rule maps are not redundant.

\subsection{n-Rule maps and rules as similarities}\label{sec:similarity}

Let $L_1, L_2$ denote lines in $\mathbb{R}^2$, intersecting at point $z$ with acute or right angle $\delta$.  Let $x, y \in L_1$ lie on the same side of the line intersection point $z$, and take a rule $r$ that projects onto line $L_2$ with projection angle $\theta \in (0, \pi/2]$.  Further, let the orientation value of $r$ be chosen so that it maps $x$ and $y$ farthest from $z$  when $\delta$ is acute (see Figure \ref{fig4}, where orientation $o=1$ projects $x$ \& $y$ farthest from $z$).

Let $d$ be the Euclidean metric.  Assume $d(x, y) = \epsilon > 0$, and $a = d(z, x)$, $a + \epsilon = d(z, y)$.  Let $\gamma = \pi - \delta - \theta$, so that $\gamma = \angle zxr(x) = \angle zyr(y)$.  Then it follows by use of the law of sines that 
\[
d(r(x), r(y)) = \left |\left| \frac{a\sin(\gamma)}{\sin(\theta)} - \frac{(a+\epsilon) \sin(\gamma)}{\sin(\theta)}  \right|\right|_2 = \frac{\epsilon \sin(\gamma)}{\sin(\theta)}.
\]
Let $c = \sin(\gamma)/\sin(\theta)$, and hence $d(r(x), r(y)) = cd(x, y)$.  We see that if $0 < \theta < (\pi-\delta) / 2$ and $\delta$ is acute, then
\[
c = \frac{\sin(\gamma)}{\sin(\theta)} > 1.
\]
Furthermore, if $\theta = (\pi-\delta)/2$, then $c=1$, and when 
\[
\frac{\pi-\delta}{2} < \theta \leq \frac{\pi}{2}
\]
then $0 \leq c < 1$.  

By similar argument, when the rule $r$ has opposite orientation parameter (and hence maps $x$ and $y$ closer to $z$ in this case), we see that $d(r(x), r(y)) = cd(x, y)$ for some constant $c = c(\theta, \delta)$, computable using the law of sines, where $0 \leq c <1$ when $\delta/2 < \theta \leq \pi/2$, and $c = 1$ when $\theta= \delta/2$, and $c > 1$ when $0 <\theta < \delta / 2$.

When the points $x$ and $y$ are on opposite sides of the line intersection point $z$, it still holds that $d(r(x), r(y)) = cd(x, y)$.  To see this, let $x, y \in L_1$ lie on opposite sides of line intersection point $z$.  Then $d(r(x), z) = cd(x, z)$ and $d(r(y), z) = cd(y, z)$, and hence 
\[
d(r(x), r(y)) = d(r(x), z)+d(r(y), z) = cd(x, z) + cd(y, z) = cd(x, y)
\]As such, by fixing the projection angle and orientation parameters $\theta$ and $o$ of a rule $r$, and restricting the mapping of a rule from one line to another $X_m$, then $r$ becomes a similarity transformation.  That is 
\[
d(r(x), r(y)) = cd(x, y), \text{ } c\geq 0.
\]
We call the constant $c$ a \emph{similarity coefficient}.

Iteration of $n$-rule maps is defined to be a cycling composition of rules in an associated rule sequence, and as a consequence, after the first iteration of a $n$-rule map over a point in $X_m$, each rule in the rule sequence will always map between the same pair of lines since each rule in the sequence always projects onto the same line.  Hence, by way of the above analysis, iteration of a fixed $n$-rule map can be thought of as a cycling composition of similarity transformations, after the first iteration of the map.

Let $\hat{T}_n \coloneqq T_n^n$ denote the induced map of $n$-rule map $T_n$, so that $\hat{T}_n^l = T_n^{ln}$ for $l \in \mathbb{N}$, and $\hat{T}_n: L_{a_n} \rightarrow L_{a_n}$, where $L_{a_n}$ is the line the $n$th rule in the defining rule sequence of $T_n$ maps onto.  If the rules defining $T_n$ have similarity coefficients $c_1,..., c_n$, then let $C = c_1\cdot c_2 \cdots c_n$ label the \textit{similarity coefficient for the induced map} $\hat{T}_n$.

\section{n-Rule maps and closed curves}\label{sec:closed}

In this section we prove Theorem \ref{thm1}.  We begin by establishing the following.

\begin{figure}
    \centering
    \includegraphics[scale=.3]{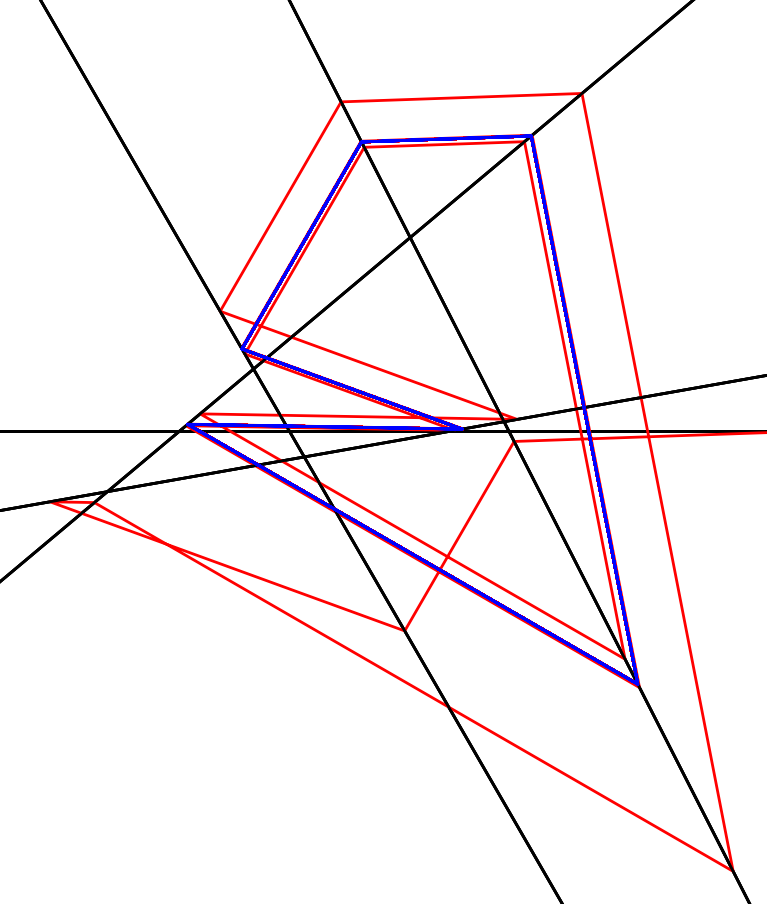}
    \caption{Result from a numerical simulation of iterating a $6$-rule map in $X_{4}$, which exhibits an orbit (red) converging to a six-cycle (blue).}
    \label{fig5}
\end{figure}

\begin{theorem}\label{thm2}
Let $(X_m, T_n)$ be a dynamical system, and let $\hat{T}_n$ be the induced map of $n$-rule map $T_n$, with similarity coefficient $C = c_1c_2\cdots c_n$.  If $0 \leq C<1$, then $T_n$ admits a unique periodic orbit of period $n$.  
\end{theorem}
\begin{proof}
Let $\hat{\mathcal{O}}(x)$ denote the orbit of $x \in X_m$ under $\hat{T}_n$.  It follows from definition of the induced map $\hat{T}_n$, that for any $x \in X_m$, the orbit $\hat{\mathcal{O}}(x)\setminus \{x\}$ of $x$ under $\hat{T}_n$, must be a subset of some line $L_{a_n} \subset X_m$ determined by the $n$th rule in the defining rule sequence of $T_n$.  Further, by hypothesis $0 \leq C<1$, and hence
\[
d(\hat{T}_n(x), \hat{T}_n(y)) \leq Cd(x, y)
\]
for any $x, y \in L_{a_n}$, so $\hat{T}_n$ is a contraction mapping.  But the line $L_{a_n}$ is a closed subset of $\mathbb{R}^2$ and necessarily complete.  Hence, by the contraction mapping theorem there exists a unique $x^* \in L_i$ such that $\hat{T}_n(x^*) = x^*$.  Then $T_n$ admits a unique periodic orbit of period $n$.
\end{proof}

Refer to Figure \ref{fig5} for visual demonstration of the type of dynamics Theorem \ref{thm2} provides.

Assume two lines $L_i, L_j$ in $X_m$ intersect at angle $\delta \leq \pi/2$.  Then if the defining rule sequence of an $n$-rule map $T_n$ over $X_m$ contains a rule mapping from $L_i$ to $L_j$ (or vice-versa) with projection angle $\theta = \delta$, then depending on rule orientation, iteration of this rule may project onto the intersection point of $L_i$ and $L_j$ every $n$ iterations, and hence the system collapses to a periodic orbit after at most $n$ iterations of $T_n$.  In such a case, we say the $n$-rule map $T_n$ is \emph{collapsing}.  Collapsing maps correspond with the case when the induced map of $n$-rule map $T_n$ has similarity coefficient $C=0$.  Figure \ref{fig3} gives a visual example of a collapsing map.

\begin{remark}\label{rem2}
If an $n$-rule map is not collapsing, then it is invertible.  
\end{remark}

\begin{figure}
    \centering
    \includegraphics[scale=.3]{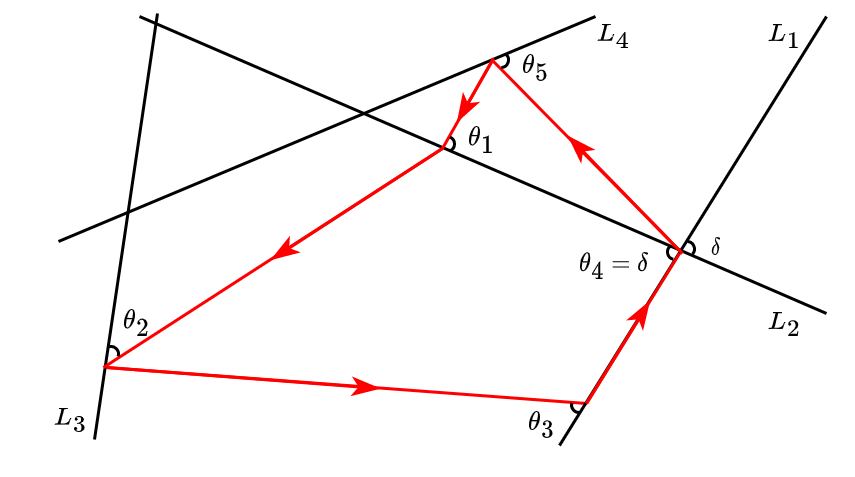}
    \caption{The closed curve corresponds to a periodic orbit generated by a collapsing $5$-rule map.  The map is composed of five rules, where $\theta_4=\delta$, the acute line intersection angle of $L_1$ and $L_2$.  And hence because the fourth rule maps from $L_1$ to $L_2$, with projection angle $\delta$ and orientation $1$, the intersection point of $L_1$ and $L_2$ is mapped onto, becoming a periodic point.}
    \label{fig3}
\end{figure}

If $m'$ lines in $X_m$ intersect at a common point $z$, with $2 \leq m' < m$, then a rule sequence $\{r_i\}_{i=1}^n$ may contain a subsequence of consecutive rules which map strictly between the $m'$ lines intersecting at $z$.  As a consequence, such a subsequence of consecutive rules would map $z$ to itself.  In the case that such a line intersection point $z$ is a periodic point, and a subsequence of rules map over this point, we say the periodic point $z$ is \textit{absorbing}, and that the subsequence of rules is an \textit{absorbed} subsequence.  For visual example, refer to Figure \ref{figAbsorbing}, demonstrating how a subsequence can be absorbed over a line intersection point (left), compared to no absorption (right).

An absorbed subsequence may be composed of $1 \leq k \leq n-2$ rules.  That $k \leq n-2$ is given by the nonconcurrency assumption of the lines composing $X_m$, and that every line in $X_m$ must be mapped onto by at least one rule in every rule sequence associated to an $X_m$.  Hence, there are always at least two rules in a rule sequence that cannot be absorbed.

\begin{figure}
    \centering
    \includegraphics[scale=.35]{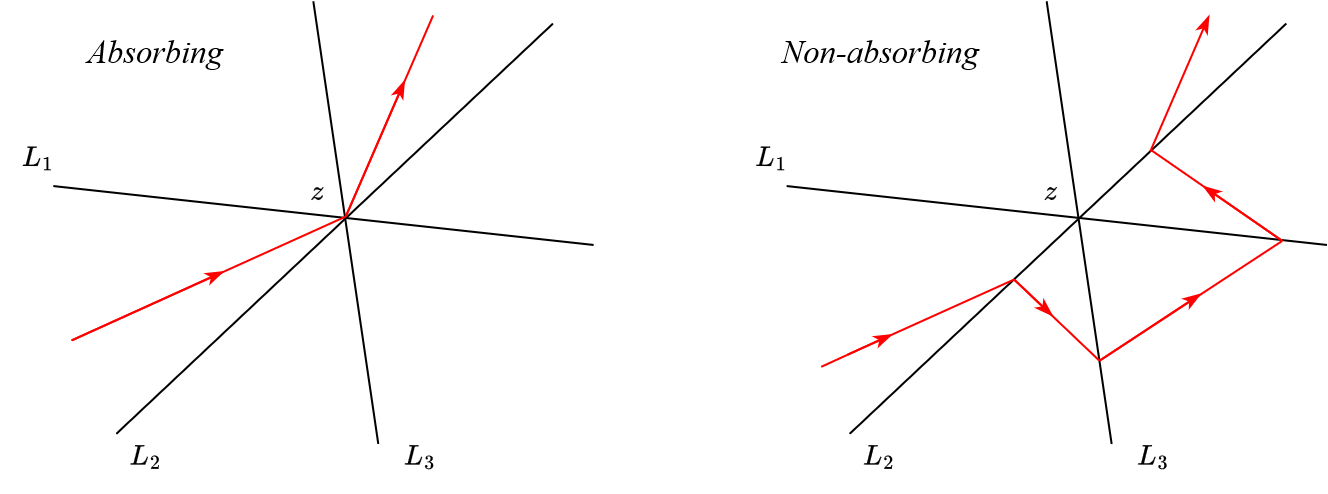}
    \caption{If a rule mapping onto line $L_2$ hits the intersection point $z = L_1\cap L_2 \cap L_3$, then the three following rules in the sequence mapping onto lines $L_3, L_1, L_2$, respectively, are ``absorbed."}
    \label{figAbsorbing}
\end{figure}

We are now in a position to prove Theorem \ref{thm1}.

\begin{proof}[Proof of Theorem \ref{thm1}]
Take a space $X_m$, a sequence $\theta_1, \theta_2,...,\theta_n$ of acute or right angles, as well as a sequence $L_{a_1}, L_{a_2}, ..., L_{a_n}$ of line labels over $X_m$ with no two consecutive labels the same, and each possible label occurring at least once in the sequence.

Fix a binary string $s \in \{0, 1\}^n$.
Consider the case when the angles $\theta_1, ..., \theta_n$ and orientation values $o_i= s_i$ of the $n$-rule map are such that the similarity coefficient for the induced map is $C<1$.  Then by Theorem \ref{thm2}, $T_n$ admits a periodic orbit, and by joining consecutive periodic points of this orbit with line segments, we obtain a closed curve $\Gamma$ over $X_m$ that admits an incidence angle sequence $\theta_1, \theta_2, ..., \theta_n$ with respect to the line sequence $L_{a_1}, L_{a_2}, ..., L_{a_n}$, by the definition of $T_n$.  Similarly, if $T_n$ is collapsing, then it has a periodic orbit, and this orbit corresponds to a closed curve $\Gamma$ admitting an incidence angle sequence $\theta_1, \theta_2, ..., \theta_n$ with respect to the line sequence $L_{a_1}, L_{a_2}, ..., L_{a_n}$.

If the $n$-rule map parameters $\{\theta_i\}$ and $\{o_i\}$ are such that $C>1$, then $T_n$ is not collapsing, and  has an inverse $n$-rule map $T^{-1}_n$ with corresponding induced map $\hat{T}_n^{-1}$ and similarity coefficient $C' =1/C$.  Then by Theorem \ref{thm2} $T_n^{-1}$ has a periodic orbit.  But $T^{-1}_n$ is the inverse of $T_n$, and hence this is also a periodic orbit for $T_n$.  Constructing a closed curve $\Gamma$ from this periodic orbit as above, we see that $\Gamma$ admits an incidence angle sequence $\theta_1, \theta_2, ..., \theta_n$ with respect to the line sequence $L_{a_1}, L_{a_2}, ..., L_{a_n}$.

Consider the case when parameters $\{\theta_i\}$ and $\{o_i\}$ give $C=1$.  Then iteration of the map does not converge to a periodic orbit.  But there must exist $\theta_i$ such that for every $\epsilon > 0$, the identical $n$-rule map with rule $r_i$ having projection angle $\theta_i \pm \epsilon$, has similarity coefficient $C \neq 1$, so one of the above cases holds.  But $\epsilon$ is arbitrary, so clearly the set of $(\theta_1,...,\theta_n) \in (0,\pi/2]^n$ that correspond to an $n$-rule map with $C=1$ with respect to a fixed line label sequence has Lebesgue measure zero.

\begin{remark}
    In many (but not all) cases, it is sufficient to fix any $i \in [n]$, and change the orientation value $o_i$ of the rule $r_i$ to $\overline{o_i}$, where $\overline{o_i} = o_i \oplus 1$, to obtain a new $n$-rule map with $C \neq 1$.
\end{remark}

Finally, consider the degenerate case when the orbit converges to a periodic orbit with an absorbing periodic point.  
Recall that given a line label sequence $L_{a_1},...,L_{a_n}$ with no consecutive labels the same, any absorbed subsequence of rules cannot have length greater than $n-2$, implying there are always at least two rules that cannot be absorbed.
If $z \in X_m$ is an absorbing periodic point, let rule $r_i$ be that non-absorbed rule mapping onto $z$. Fix any $\epsilon_i > 0$ and update the $n$-rule map so that rule $r_i$ now has projection angle $\theta_i' = \theta_i \pm \epsilon_i$.
The next rule $r_{i+1}$ is now clearly no-longer absorbed: if it were, $r_i$ would map onto $z$ with both projection angles $\theta_i$ and $\theta_i'$, an impossibility since it must be mapping from the same point.  If rule $r_{i+2}$ is still absorbed over $z$, repeat this process, which clearly halts.
The $\epsilon_i$ are arbitrary, so it follows that the projection angle parameters $(\theta_1,...,\theta_n) \in (0,\pi/2]^n$ that give degenerate closed curves with respect to a fixed line label sequence have Lebesgue measure zero.

The angle parameter sets corresponding to the $C=1$ and degenerate absorbing cases have zero measure in $(0,\pi/2]^n$, as does their union.
\end{proof}

The dynamics of ``classical billiards" in which the billiard undergoes specular reflection at the boundary is an archetypal example of conservative dynamics.  Classical billiards then fail to model phenomena that hold in regimes far from equilibrium.  In the direction of overcoming these restrictions, in \cite{chernov1993steady} the authors consider a dynamical system which corresponds to motion of a single particle reflecting off scatterers, but so that the particle is subjected to an electric field and a momentum dependent frictional force, so the Liouville measure is not preserved.  In a similar direction, a number of recent papers (e.g. \cites{markarian2010pinball, arroyo, arroyo2, gtroub}) have studied the dynamics of a type of non-conservative ``pinball billiards" where the ball is ``kicked" by the wall, giving a new impulse in the direction of the normal.  That is, the outgoing angle from a collision is a uniform contraction by a factor $\lambda \leq 1$, where the classical Hamiltonian case of elastic collisions is when $\lambda$ = 1.  For $\lambda<1$ the dynamics is dissipative, and thus gives rise to attractors.

Similarly, in the case of the dynamical system studied in this paper, the system resembles a particle that reflects off or refracts through a boundary as a function of the section of the boundary it is incident to.  
In fact, the role of $\lambda$ in pinball billiards papers is nearly identical to the role of similarity coefficient $C$ in our paper: $C=1$ represents the Hamiltonian case just as $\lambda=1$, and when $C< 1$ the dynamics give rise to attractors, just as in the $\lambda< 1$ case.  Indeed, the methods of studying classical polygonal billiards by way of pinball billiards can be carried over to study classical billiards using $n$-rule maps in a similar way, although additional analysis is needed along with the techniques used in proving Theorem \ref{thm1} to actually obtain classical billiard orbits.

For instance, in \cite{gtroub} the authors introduce the notion of $\lambda$-stability, in which polygonal billiard periodic trajectories can be classified as $\lambda$-stable if there is a periodic trajectory from the corresponding pinball billiard which converges to it as $\lambda\rightarrow 1$.  In particular, the billiard of any polygon can be embedded in a one-parameter family, parameterized by $\lambda \in (0, \infty)$, of billiards having periodic trajectories whenever $\lambda \neq 1$.   
The authors use these facts to characterize the $\lambda$-stable periodic trajectories of the billiard in $P$.

In our case, we obtain the conservative limit as $C\rightarrow 1$.  It is easy to show that if there exists a periodic billiard trajectory in a polygon, then there exists a closed curve generated by an $n$-rule map with $C\neq 1$ approximating the periodic billiard trajectory.  To this end, a natural and important question that arises would be to determine when an $n$-rule map with $C\neq 1$ has a periodic orbit contained strictly within a polygon cut out by the lines composing $X_m$.

Tiling billiards are also a recently studied type of billiards, in which trajectories refract through planar tilings, with positive and negative indices of refraction \cites{davis, davis2}.  We obtain a similar physical interpretation in our case: consider a cracked pane of glass that a beam of light shines through the edge of.  If various materials are allowed between the edges of the broken glass, or the fragments of glass themselves are of different types, the beam of light will reflect off or refract through the pieces in various ways that could lead to asymptotically stable behavior.  $n$-rule maps provide a natural tool to determine the asymptotic behavior of such a system.

\section{n-Rule maps defined using piecewise continuous rules}\label{sec:dist}

This section begins the second part of our analysis, in which we redefine $n$-rule maps so that rules map onto lines not on a basis of some fixed line label, but rather on a basis of a distance between points and lines.  Rules then become piecewise continuous, and this redefinition introduces discontinuities that complicate the dynamics. In this section we give basic results concerning the redefined $n$-rule maps, and then in Section \ref{sec:periodic} we prove Theorem \ref{thm3}, which shows when their orbits are asymptotically periodic.  

The following dynamical system most directly generalizes the system studied in \cite{everett} and shares characteristics with piecewise isometric dynamical systems \cites{goetz1, goetz2}.  However,  perhaps the most closely related dynamical system is that coming from the \emph{farthest point map}, defined as follows.  If $X$ is a compact metric space, the farthest point map $f$ is defined so that for $p \in X$, $f(p)$ is the the set of all points $q$ such that the distance from $p$ is maximized at $q$.  Typically, $f$ is single valued, so we obtain a well-defined map with which we can construct a dynamical system.  The farthest point map is a well studied dynamical system; see \cite{schwartz2022farthest} and \cite{schwartz2021farthestpointmapregular} for recent work on the regular octahedron and dodecahedron, and \cite{rouyer2005antipodes, rouyer2010antipodes,wang2020farthest} for study of the farthest point map on convex polyhedron, and \cite{vilcu2006properties} for a survey.

\subsection{Redefining n-rule maps}
Let $Y_m\subset \mathbb{R}^2$ label the space of $m \geq 3$ pairwise nonparallel, nonconcurrent line in $\mathbb{R}^2$.  If $L_i, L_j$ are two distinct lines in $Y_m$ where $x \in L_i$, we define
\[
d(x, L_j) \coloneqq \inf\{d(x, y)|y \in L_j\}
\]
to be the distance between point $x$ and line $L_j$, where $d$ is the Euclidean metric, and, in particular $d(x, L_i)=0$.

For every point $x \in Y_m$, we construct the set 
\[
D(x) = \{d(x, L_i) | 1 \leq i \leq m\}
\]
and define the partially ordered set $\mathcal{D}(x) \coloneqq (D(x), \leq)$.  We let $l$ denote an index on $\mathcal{D}(x)$, so that $l=i$, $1 \leq i \leq m$ corresponds with the $i$th farthest line from point $x$.  Note that if there exist $m'$ lines in $Y_m$, $m' < m$, that are all the same distance from point $x$, then there are $m'$ values of $l$ that do not correspond with a unique distance value in $\mathcal{D}(x)$, and thus do not correspond with a unique line in $Y_m$.

\begin{definition}
Let $\theta \in (0, \pi/2]$, $o \in \{0, 1\}$, and $l \in \{2, 3, ..., m\}$.  We call a mapping $r:Y_m \rightarrow Y_m$ a \emph{piecewise rule}, if $r(x)$ is an angle $\theta$, orientation $o$ projection $x \in Y_m$ onto the $l$th farthest line from $x$.  If the index $l$ corresponds with a distance value in $\mathcal{D}(x)$ that is not unique in $\mathcal{D}(x)$, put $r(x) = x$.
\end{definition}

As before, we notate piecewise rules as $r(x; \theta, o, l)$ to make the parameters of the rule explicit.  Furthermore, for the rest of the paper, we will refer to ``piecewise rules" as ``rules" for convenience, and when needed refer to the rules used in Sections \ref{sec:prelim} and \ref{sec:closed} as ``symbolic rules."

In the case when the index $l$ corresponds to a distance value in $\mathcal{D}(x)$ that is not unique so that $r(x) = x$, then we say $x$ is an \emph{invariant} point under rule $r$. Figure \ref{fig1} provides visual demonstration of the composition of two rules, 
\[
r_1(x) \coloneqq r(x; \theta_1, 1, 2) \text{ and } r_2(x) \coloneqq r(x; \theta_2, 0, 3)
\] 
over a point $x_0 \in Y_3$, so that $r_1(x_0) = x_1$ and $r_2(r_1(x_0)) = x_2$.  Intuitively, rule $r_1$ maps to the \textit{closest} line from a point $x$, and rule $r_2$ maps to the \textit{farthest} line from a point $x$ when applied to a space $Y_3$.  We leave rule orientation to be defined as previously, with the convention made visually explicit for this new class of rules in Figures \ref{fig1} and \ref{fig4}.

We leave rule sequences and $n$-rule maps defined as before, except for noting that $n$-rule maps defined by piecewise rules may only have one rule in the defining rule sequence, unlike those defined with symbolic rules which required at least three.  We let $K_n$ denote an $n$-rule map where the defining rules in the rule sequence are piecewise rules.  We may call such maps \textit{piecewise $n$-rule maps} for clarity, although for the remainder of this paper we will only work with piecewise $n$-rule maps, and hence we refer to them simply as ``$n$-rule maps" when the context is clear.  

For any piecewise $n$-rule map $K_n$, it is required that at least one of the rules in the associated rule sequence has index value $l>2$; such a restriction ensures the dynamics of a piecewise $n$-rule map are nontrivial.  If all rules of the rule sequence have $l$ index value of $l=2$, then each iteration maps to the ``closest" line, and orbits approach a line intersection point of $Y_m$, failing to exhibit behavior of interest.

Unless otherwise stated, the pair $(Y_m, K_n)$ denotes a dynamical system.  

If a point $x^* \in Y_m$ is invariant for $n'<n$ of the rules in the $n$-rule sequence defining $K_n$, then we say $x^*$ is \emph{sometimes invariant} under $K_n$.  If $x^*$ is invariant under all rules defining $K_n$, we say $x^*$ is \emph{strictly invariant} under $K_n$.  As such, any 1-rule map has only strictly invariant points.

\begin{remark}
For any $(Y_m, K_n)$ dynamical system, the set of strictly invariant and sometimes invariant points is finite.
\end{remark}

As before, we call an $n$-rule map \emph{redundant} if there exists a length $n'$ rule sequence with $1 \leq n' < n$, such that for all $x \in Y_m$, the orbit of $x$ under the $n'$-rule map is equal to the orbit of $x$ under the $n$-rule map.  We assume all piecewise $n$-rule maps are not redundant.

\subsection{Convergence and contraction of piecewise n-rule maps}  We now give results pertaining to piecewise $n$-rule maps that are used in determining asymptotic behavior of orbits.

The space $Y_m$ is composed of $m$ pairwise nonparallel, nonconcurrent lines, so any given space $Y_m$ has $\binom{m}{2}$ pairwise line intersection points.  Then, for each pairwise intersection point, let $\eta_i$ label the $i$th pairwise line intersection angle, where $0< \eta_i \leq \pi/2$.  Let 
\[
\delta = \min\left\{\eta_i | 1 \leq i \leq \binom{m}{2}\right\}
\]
label the least pairwise intersection angle between any two lines in $Y_m$.  Note $\delta$ must be acute by definition of $Y_m$.

\begin{definition}[Average Contraction Condition]
For piecewise $n$-rule map $K_n$, let $\bar{\theta}$ label the average of all projection angles in the $n$-rule sequence defining $K_n$.  Then if 
\begin{equation}\label{eq1}
\frac{\pi-\delta}{2} < \bar{\theta} \leq \frac{\pi}{2}
\end{equation}
for least angle $\delta$ in $Y_m$, we say $K_n$ satisfies the \emph{average contraction condition} with respect to $Y_m$.
\end{definition}

We motivate the introduction of the average contraction condition through the following observations, which are similar to those given in Section \ref{sec:similarity}.

Let $L_1, L_2$ denote lines in $\mathbb{R}^2$, intersecting at point $z$ with acute angle $\delta$.  Without loss of generality, let $x, y \in L_1$, and take a rule $r$, such that $r(x), r(y) \in L_2$, and $x, y, r(x), r(y)$ are on the same side of intersection point $z$.  Further, let the orientation value of $r$ be the choice that maps farthest from $z$.  For example, in Figure \ref{fig4} rule orientation value $o=1$ maps farther away from the line intersection point when mapping from the particular line.

Assume $d(x, y) = \epsilon > 0$, and $a = d(z, x)$, $a + \epsilon = d(z, y)$.  Let $\theta$ denote the projection angle of rule $r$, and let $\gamma = \pi - \delta - \theta$, so that $\gamma = \angle zxr(x) = \angle zyr(y)$.  Then if 
\[
0 < \theta < \frac{\pi-\delta}{2}
\]
it follows by use of the law of sines that 
\[
d(r(x), r(y)) = \left|\left| \frac{a\sin(\gamma)}{\sin(\theta)} - \frac{(a+\epsilon) \sin(\gamma)}{\sin(\theta)}  \right|\right|_2 = \frac{\epsilon \sin(\gamma)}{\sin(\theta)}
\]
but $\theta < (\pi-\delta) / 2$ and $\delta$ is acute, so under our choice of rule orientation value
\[
\frac{\sin(\gamma)}{\sin(\theta)} > 1
\]
and then $d(r(x), r(y)) > d(x, y)$: iteration of $r$ over $L_1$ and $L_2$ in such a way is then expansive.  By similar argument, we see that if $\theta = (\pi-\delta)/2$, then the rule defines an isometry and $d(r(x), r(y)) = d(x, y)$.  When 
\[
\frac{\pi-\delta}{2} < \theta \leq \frac{\pi}{2}
\]
then $d(r(x), r(y)) \leq cd(x, y)$, $0 \leq c < 1$.  Further, $\delta$ is acute, so in the case when $r(x)$ maps opposite the angle $\pi - \delta$, we have $\pi-\delta > \delta$, so if $r$ is contractive when mapping opposite $\delta$, it must also be contractive when mapping opposite $\pi-\delta$.  

From the above example, we see that for any rule $r$, with colinear $x, y$ and colinear $r(x), r(y)$ all on the same side of the line intersection point, we have
\[
d(r(x), r(y)) \leq cd(x, y), \text{ } c\geq 0
\]
where $c$ can be computed directly via the law of sines, as a function of the rule projection angle and the opposite line intersection angle.  In this case, we call such values $c$, \emph{separation coefficients}.

\begin{lemma}\label{lem7}
Let lines $L_1, L_2 \subset Y_m$ intersect at a point $z$ with acute angle $\delta$, and let $x, y \in L_1$ lie on the same side of $z$.  Let $r_1, r_2$ label rules with distinct orientation values, which are chosen so that the rules map farthest from $z$, and let the points $r_i(x), r_i(y) \in L_2$ and $r_i(r_j(x)), r_i(r_j(y)) \in L_1$, $i, j=1, 2$, $i \not=j$ all lie on the same side of $z$.  Then for corresponding rule separation constants $c_1$ and $c_2$, we have that $0 \leq c_1c_2 < 1$ if and only if
\begin{equation}\label{eq3}
    \frac{\pi-\delta}{2} < \frac{\theta_1 + \theta_2}{2} \leq \frac{\pi}{2}
\end{equation}
for rule projection angles $\theta_1, \theta_2$ corresponding with rules $r_1, r_2$.
\end{lemma}

Note that $c_1c_2<1$ implies composition of the two rules defines a contraction:
\[
d(r_i(r_j(x)), r_i(r_j(y))) \leq c_1c_2d(x, y),\text{ } 0 \leq c_1c_2<1
\]
for distinct $i, j$.  Further, the orientation values of the rules are chosen so that the rules map farthest from the line intersection point in each case, and thus the corresponding separation constants are maximized.  

\begin{proof}[Proof of Lemma \ref{lem7}]
Let $\gamma_1 = \pi - \theta_1 - \delta$ and $\gamma_2 = \pi - \theta_2 - \delta$.  We assume that
\[
\frac{\pi-\delta}{2} < \theta_1 \leq \frac{\pi}{2}
\]
so by Equation \ref{eq3}, we require that
\begin{equation}\label{eq4}
    \pi-\theta_1-\delta < \theta_2 \leq \pi - \theta_1
\end{equation}
By Equation \ref{eq4} we see that $\sin(\theta_2) > \sin(\pi-\theta_1-\delta)$, and that
\begin{align*}
    \sin(\gamma_2) &= \sin(\pi-\theta_2-\delta) \\
    & = \sin(\theta_2 + \delta) \\
    & \leq \sin(\pi-\theta_1+\delta)
\end{align*}
Then, through substitution we obtain
\[
c_1c_2 = \frac{\sin(\gamma_1)}{\sin(\theta_1)}\frac{\sin(\gamma_2)}{\sin(\theta_2)} 
    < \frac{\sin(\pi-\theta_1-\delta)}{\sin(\theta_1)}\frac{\sin(\pi-\theta_1+\delta)}{\sin(\pi-\theta_1-\delta)} 
\]
but $\sin(\pi-\theta_1 +\delta) = \sin(\theta_1 - \delta) < \sin(\theta_1)$, so 
\[
\frac{\sin(\pi-\theta_1-\delta)}{\sin(\theta_1)}\frac{\sin(\pi-\theta_1+\delta)}{\sin(\pi-\theta_1-\delta)}  = \frac{\sin(\pi-\theta_1+\delta)}{\sin(\theta_1)} < 1
\]
Going the other direction, let $\gamma_1 = \pi - \theta_1 - \delta$ and $\gamma_2 = \pi - \theta_2 - \delta$.  Then from 
\[
c_1c_2 = \frac{\sin(\gamma_1)\sin(\gamma_2)}{\sin(\theta_1)\sin(\theta_2)} < 1
\]
with substitution we obtain
\[
\sin(\pi-\theta_1-\delta)\sin(\pi-\theta_2-\delta) < \sin(\theta_1)\sin(\theta_2)
\]
By the product identity for sine, we have
\[
\frac{\cos(-\theta_1+\theta_2)-\cos(2\pi-\theta_1-\theta_2-2\delta)}{2} 
    < \frac{\cos(\theta_1-\theta_2)-\cos(\theta_1 + \theta_2)}{2} 
\]
but $\cos(-1(\theta_1-\theta_2)) = \cos(\theta_1-\theta_2)$ so upon simplifying we have
\[
\cos(2\pi-\theta_1-\theta_2-2\delta) > \cos(\theta_1+\theta_2)
\]
and by removing cosine we obtain
\begin{align*}
    2\pi-\theta_1-\theta_2-2\delta &< \theta_1+\theta_2
\end{align*}
Note that although cosine is not monotone, by the restrictions on the angles we can remove cosine in such a way.  This gives us
\[
\pi-\delta < \theta_1 + \theta_2 \Longrightarrow \frac{\pi-\delta}{2} < \frac{\theta_1 + \theta_2}{2}
\]
The case for the upper bound $(\theta_1+\theta_2)/2 \leq \pi/2$ is clear.
\end{proof}

In the above lemma, we took the rule orientation values to be chosen in a way that ensures the rules map farthest from the line intersection point.  If we instead use rule orientation values that force the rules to map closer to the line intersection points, then so long as the two projection angles $\theta_1, \theta_2$ satisfy Equation \ref{eq3}, \textit{both rules} must provide contraction (i.e. $c_1<1$ and $c_2<1$).  

That is, if $\gamma_2 = \pi-\theta_2-\delta$, and
\[
\frac{\sin(\gamma_2)}{\sin(\theta_2)} > 1
\]
under a rule orientation value mapping farther from a line intersection point, then under opposite rule orientation value, we have $\gamma_2' = \pi - (\pi - \theta_2) - \delta$, so
\[
\frac{\sin(\gamma_2')}{\sin(\pi - \theta_2)} = \frac{\sin(\theta_2-\delta)}{\sin(\theta_2)} < 1.
\]
As an immediate consequence of the above remark and Lemma \ref{lem7}, we obtain the following corollary.

\begin{corollary}\label{cor2}
Let $L_1, L_2 \subset Y_m$ intersect at point $z$ with acute angle $\delta$.  Further, let $K_n$ be a piecewise $n$-rule map so that iterates of $K_n$ map between $L_1$ and $L_2$, opposite angle $\delta$, and for initial points $x, y \in L_1$, let the first $n$ points of $\mathcal{O}(x), \mathcal{O}(y)$ remain on the same side of $z$.  Then if $K_n$ satisfies the average contraction condition for least angle $\delta$,
\[d(K^n_n(x), K^n_n(y)) \leq Cd(x, y), \text{ } 0 \leq C <1.\]
\end{corollary}

We remark that here $C = c_1c_2\cdots c_n$, is the product of the $n$ separation constants coming from the piecewise rule sequence.

\begin{lemma}\label{lem6}
For all lines $L_i \subset Y_m$ and $x, y \in L_i$, if each closed interval
\[
[K_n^i(x), K_n^i(y)] \subset Y_m, \text{ } 0 \leq i \leq n
\]
contains no line intersection points or invariant points, then if $K_n$ satisfies the average contraction condition in $Y_m$,
\[
d(K^n_n(x), K^n_n(y)) \leq Cd(x, y), \text{ } 0 \leq C < 1.
\]
\end{lemma}
\begin{proof}
Let $\delta$ label the least pairwise intersection angle in $Y_m$.  Then by Corollary \ref{cor2}, if $K_n$ satisfies the average contraction condition over $Y_m$, and iteration of $K_n$ is strictly opposite angle $\delta$, then $d(K^n_n(x), K^n_n(y)) \leq Cd(x, y)$ for $C \in [0, 1)$.  But $\delta$ is the least angle in $Y_m$, so if iteration of $K_n$ contracts opposite angle $\delta$ on average, then it must also contract opposite every other angle in $Y_m$ on average: if $\eta_i$ is a distinct line intersection angle, then $\eta_i \geq \delta$, and
\[
\frac{\pi-\eta_i}{2} \leq \frac{\pi-\delta}{2}.
\]
As such, assuming the conditions of the statement, it follows that
\[d(K^n_n(x)K^n_n(y)) \leq Cd(x, y)\] for $C \in [0, 1)$.
\end{proof}

We note the average contraction condition ensures contraction regardless of rule orientation.  The average contraction condition provides sufficient but not necessary conditions for an $n$-rule map to define a contraction on average. 

\begin{lemma}\label{lem1}
If $K_n$ satisfies the average contraction condition over $Y_m$, then there exists bounded regions $R, R' \subset Y_m$ such that for all $x \in R$, $\overline{\mathcal{O}(x)} \subset R'$.
\end{lemma}
\begin{proof}
By definition of $n$-rule maps and the average contraction condition, iteration of an $n$-rule map $K_n$ in $Y_m$ must, on average, map closer to line intersection points.  The lines composing $Y_m$ are pairwise nonparallel, so all lines must intersect, and there must exist a bounded region $R$ containing all such line intersection points.  As such, if iteration of $K_n$ maps closer to line intersection points on average, then iteration of the map must remain in a bounded region $R'$.
\end{proof}

Immediate from proof of Lemma \ref{lem1}, we obtain the following corollary.

\begin{corollary}\label{cor1}
If $K_n$ satisfies the average contraction condition over $Y_m$, then any sequence of points taken from successive preimages of $K_n$ over noninvariant points $x \in Y_m$ diverges in $Y_m$.
\end{corollary}

\section{Asymptotic behavior of piecewise n-rule maps}\label{sec:periodic}

In this section, we study the asymptotic properties of piecewise $n$-rule maps satisfying the average contraction condition over $Y_m$.  For piecewise $n$-rule map $K_n$ and point $x \in Y_m$, we call a \emph{cycle} of $K_n$ over $x$ the application of $K_n$ to $x$, $n$ times; the cycle of $x_0 \in Y_m$ under $K_n$ is the sequence of points $x_0, x_1, ..., x_n$, where $K_n^n(x_0) = x_n$.  We let $\mathcal{K}_n \coloneqq K_n^n$ label the \emph{cycle map} of $K_n$, so that for $x_0 \in Y_m$, $\mathcal{K}_n(x_0) = x_n$, and $\mathcal{K}^t_n(x_0) = K^{tn}_n(x_0) = x_{tn}$.

If rule $r_i$ in the rule sequence of $n$-rule map $K_n$ has sometimes invariant point $q$ in $Y_m$, then for every $h \in Y_m$ such that $K^i_n(h) = q$ for $1 \leq i \leq n$, we call $h$ a \emph{pre-invariant} point of rule $r_i$.  Associated with the $(Y_m, K_n)$ dynamical system, we let $\Omega$ denote the set of invariant points of all types, as well as preimages of the cycle map $\mathcal{K}_n$ from all pre-invariant points.  Further, if $p$ is a strictly invariant point under $K_n$, then all points $a \in Y_m$ such that $\mathcal{K}_n^j(a) = p$, $j \in \mathbb{Z}^+$, are also contained in $\Omega$.

Put $Y'_m = Y_m \setminus \Omega$.  We call the dynamical system $(Y_m, K_n)$ \emph{degenerate} when iteration of $K_n$ eventually maps to an invariant point of any type; it follows that for the dynamical system $(Y_m', \mathcal{K}_n)$ to be well defined, $(Y_m, K_n)$ must be a non-degenerate dynamical system.  Such degenerate systems arise at bifurcation points, and the remainder of this section focuses on the study of non-degenerate systems.  The main result of this section is as follows.

\begin{theorem}\label{thm3}
Let $(Y_m, K_n)$ be a non-degenerate system, with piecewise $n$-rule map $K_n$ satisfying the average contraction condition over $Y_m$.  Then there exists $k \in \mathbb{Z}^+$ such that for all $x \in Y_m'$, the orbit $\mathcal{O}(x)$ converges to a periodic orbit of period $kn$.
\end{theorem}

Figure \ref{fig:piecewise} illustrates the kind of dynamics Theorem \ref{thm3} provides, showing the periodic orbit iteration of a $4$-rule map converged to in a space $Y_{5}$.   

\begin{figure}
    \centering
    \includegraphics[scale=.55]{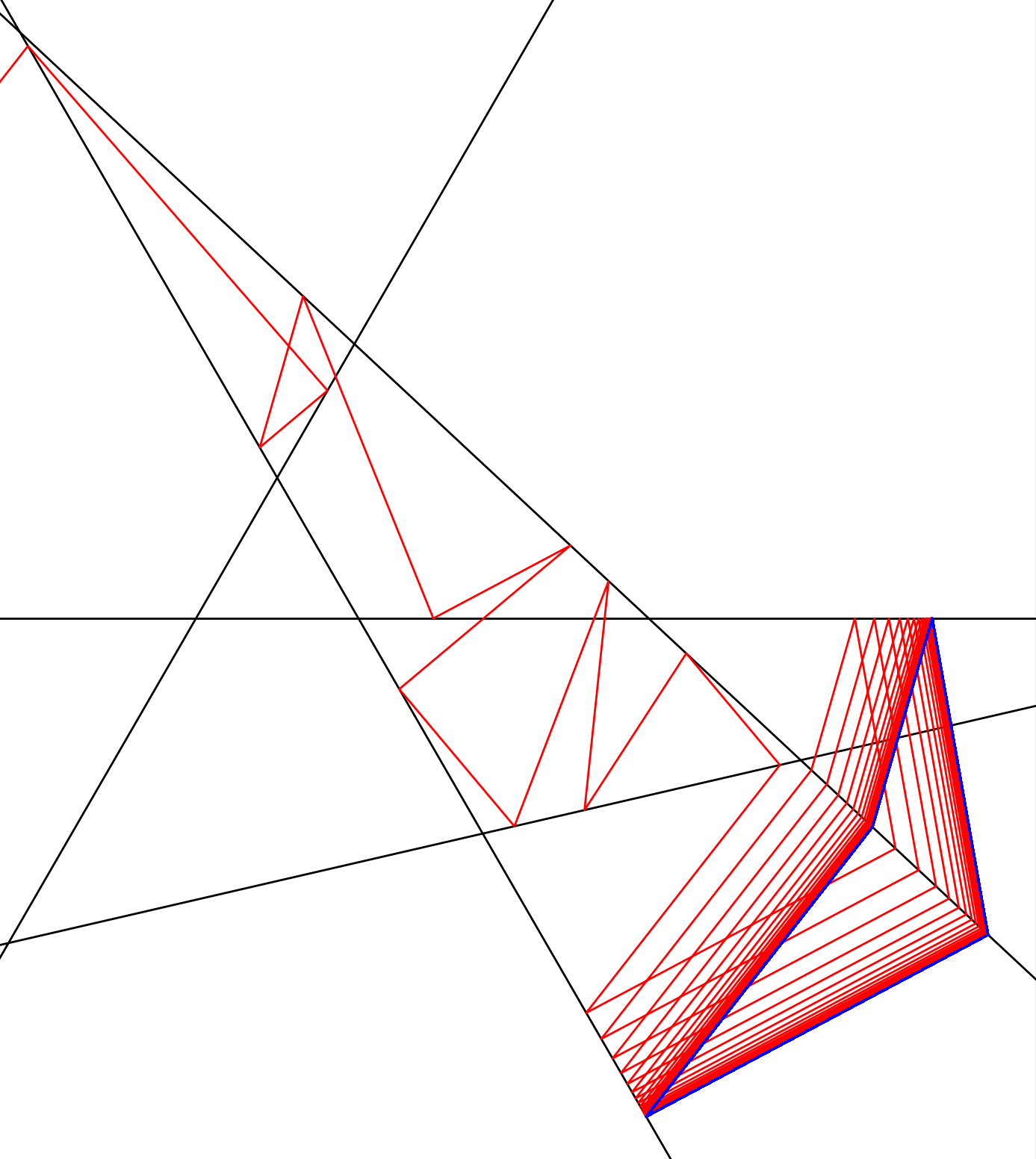}
    \caption{Iteration of a piecewise $4$-rule map in $Y_{5}$, with the orbit converging to a four-cycle.  This was the output of a numerical simulation.}
    \label{fig:piecewise}
\end{figure}

We need some preparatory lemmas to prove Theorem \ref{thm3}.  First, note that as consequence of Corollary \ref{cor1}, if $K_n$ satisfies the average contraction condition, then sequences of points taken from preimages of invariant points diverge in $Y_m$, so $\Omega$ is guaranteed not to be dense in $Y_m$ since the collection of sometimes invariant and strictly invariant points is finite.  If, however, $K_n$ fails to satisfy the average contraction condition then such a guarantee may not be made.

If $K_n$ satisfies the average contraction conditions in $Y_m$, then let $U_m$ denote the set of open intervals $I_a \subset Y_m$ such that the boundary values of each $I_a$ are given by elements in $\Omega$; no element in $\Omega$ is contained within an open interval $I_a$.  Let 
\[\hat{\mathcal{O}}(x) \coloneqq \{x, \mathcal{K}_n(x), \mathcal{K}^2_n(x),...\}\] denote the orbit of $x$ under cycle map $\mathcal{K}_n$.

\begin{lemma}\label{lem2}
For non-degenerate dynamical system $(Y_m, K_n)$ and piecewise $n$-rule map $K_n$ satisfying the average contraction condition over $Y_m$, if $I_a \in U_m$, then there exists an $I_b \in U_m$ such that $\mathcal{K}_n[I_a] \subset I_b$, where $I_a, I_b$ need not be distinct.
\end{lemma}
\begin{proof}
We proceed by contradiction and assume $\mathcal{K}_n[I_a] \subset I_{b} \cup I_{c}$.  By definition, the boundary values of each $I_a \in U_m$ are invariant points or preimages of invariant points under $\mathcal{K}_n$.  It follows that if $\mathcal{K}_n[I_a] \subset I_{b} \cup I_{c}$, then $I_a = I_{d} \cup I_{e}$, as $I_a$ would contain preimages of such boundary values, a contradiction.
\end{proof}

\begin{lemma}\label{lem3}
If $(Y_m, K_n)$ is non-degenerate with $K_n$ satisfying the average contraction condition and $x \in Y'_m$, then $\hat{\mathcal{O}}(x) \subset \bigcup_{i=1}^s I_i$ for finite $s$ and $I_i \in U_m$.
\end{lemma}
\begin{proof}
We first note that taking $x \in Y_m'$ is, by definition, equivalent to taking $x \in I_a$, for $I_a$ in $U_m$.  For any $(Y_m, K_n)$ dynamical system, there may only be a finite number of invariant points of any type under $K_n$, and by Corollary \ref{cor1}, preimages of $n$-rule maps satisfying the average contraction condition diverge from points in $Y_m$.  It then follows by definition of the set $\Omega$ and corresponding construction of intervals in $U_m$, that for any bounded region $R \subset Y_m$, there may only be a finite number of such intervals $I_a$ in $R$.  Further, by Lemma \ref{lem1}, orbits of $n$-rule maps satisfying the average contraction condition must remain in a bounded region. Finally, by Lemma \ref{lem2}, for every $I_a \in U_m$, $\mathcal{K}_n[I_a] \subset I_b$, and it thus follows that the orbit of $x$ under $\mathcal{K}_n$ is contained in a finite number of intervals.
\end{proof}

We call an interval $I_c \in U_m$ \emph{confining} if there is a $t \in \mathbb{Z}^+$, $t = t(I_c, K_n)$, such that $\mathcal{K}_n^t[I_c] \subset I_c$.  

\begin{lemma}\label{lem4}
If $(Y_m, K_n)$ is a non-degenerate dynamical system with $n$-rule map $K_n$ satisfying the average contraction conditions over $Y_m$, then there exists a confining interval $I_c$ in $Y_m$, and iteration of $\mathcal{K}_n$ over any $x \in Y_m'$ maps into a confining interval in a finite number of iterations.
\end{lemma}
\begin{proof}
By Lemma \ref{lem3}, the orbit of any $x \in Y_m'$ under $\mathcal{K}_n$ is restricted to a finite number of intervals.  Thus, by way of the pigeon hole principle, iteration of $\mathcal{K}_n$ is forced to map to an interval it has already visited in a finite number of iterations: a confining interval.  And because the orbit is restricted to a finite number of intervals, it must map into a confining interval after a finite number of iterations. 
\end{proof}

\begin{definition}\label{def1}
Let $I_c \in U_m$ be a confining interval in $Y_m$, and let $\hat{\mathcal{K}}_n:I_c \rightarrow I_c$ be the \emph{induced map} of $\mathcal{K}_n$ over the interval of continuity $I_c$, defined so that if $x \in I_c$ and $\mathcal{K}_n^{k}(x) \in I_c$ for minimal $k$, we put $\hat{\mathcal{K}}_n(x) = \mathcal{K}^{k}_n(x) = K^{kn}_n(x)$.
\end{definition}

\begin{lemma}\label{lem5}
If $(Y_m, K_n)$ is a non-degenerate dynamical system with $n$-rule map $K_n$ satisfying the average contraction condition over $Y_m$, and let $I_c$ be a confining interval in $Y_m$.  Then the induced map $\hat{\mathcal{K}}_n$ has a unique fixed point in $I_c$.
\end{lemma}
\begin{proof}
By hypothesis, $K_n$ satisfies the average contraction condition over $Y_m$, so $\hat{\mathcal{K}}_n$ is a contraction mapping over confining interval $I_c$ as consequence of Lemma \ref{lem6}.  Further, we take the system $(Y_m, K_n)$ to be non-degenerate, so by Lemma \ref{lem2}, $\hat{\mathcal{K}}_n[I_c] \subset I_c$ (strict subset).  As such, for any $x \in I_c$, the sequence of points $x, \hat{\mathcal{K}}_n(x), \hat{\mathcal{K}}^2_n(x), ...$ is a Cauchy sequence, and must converge to a unique point in the interval of continuity $I_c$.  It follows that there is a point $x^* \in I_c$ such that $\hat{\mathcal{K}}_n(x^*) = x^*$.
\end{proof}

We now prove Theorem \ref{thm3}.

\begin{proof}[Proof of Theorem \ref{thm3}]
By hypothesis, $(Y_m, K_n)$ is a non-degenerate dynamical system, $K_n$ satisfies the average contraction condition, and we take $x \in Y_m'$ so iteration of $K_n$ over $x$ does not map to an invariant point of any type.  It then follows as a consequence of Lemma \ref{lem4} that iteration of $K_n$ over $x \in Y_m'$ maps into a confining interval $I_c \in U_m$ in a finite number of iterations.  And by consequence of Lemma \ref{lem5} and Definition \ref{def1}, iteration of $K_n$ in a confining interval must converge to a periodic orbit of period $kn$, $k \in \mathbb{Z}^+$.
\end{proof}

We remark that for particular periodic orbits generated by an $n$-rule map in $Y_m$, we cannot claim that the corresponding basin of attraction is all of $Y'$, as the periodic orbit is also dependent on initial condition $x_0 \in Y'$.  Indeed, work established in \cite{nogueira} for example, which concerns piecewise contractions of the interval, motivates questions regarding upper bounds for the number of distinct periodic orbits a fixed $(Y_m, K_n)$ dynamical system can admit.  One other question that arises from our analysis is whether there are conditions that can be used to tell whether a dynamical system $(Y_m, K_m)$ is degenerate or not.

Software that can be used to simulate both types of $n$-rule maps is publicly available at \cite{Everett_N-Rule-Maps_2020}.

\bibliographystyle{abbrv}
\bibliography{references}
\end{document}